\documentclass[a4paper]{amsart}

\usepackage[mathscr]{eucal}
\usepackage{amsmath,amssymb,amsbsy,comment}
\usepackage[all]{xy}
\usepackage{hyperref}
\hypersetup{%
  colorlinks = true,
  citecolor = green,
  urlcolor = black,
  pdfauthor = {Ryuma Orita},
  pdfkeywords = {symplectic geometry, topology, mathematics, non-contractible periodic orbits, Hamiltonian flows, Floer--Novikov homology},
  pdftitle = {Ryuma Orita: Non-contractible periodic orbits in Hamiltonian dynamics on tori},
  pdfsubject = {Non-contractible periodic orbits in Hamiltonian dynamics on tori},
  pdfpagemode = UseNone,
  bookmarksnumbered=true,
  bookmarks=true
}
\usepackage[abs]{overpic}
\title[Non-contractible Hamiltonian periodic orbits in tori]%
	{Non-contractible periodic orbits in Hamiltonian dynamics on tori}

\author{Ryuma Orita} 
\address{Graduate School of Mathematical Sciences, the University of Tokyo, 3-8-1 Komaba, Meguro-ku, Tokyo 153-0041, Japan}
\email{orita@ms.u-tokyo.ac.jp}
\urladdr{https://sites.google.com/site/oritaryuma/}

\thanks{This work was supported by JSPS KAKENHI Grant Number JP02607057 and the Program for Leading Graduate Schools, MEXT, Japan.
The author was supported by the Grant-in-Aid for JSPS fellows.}

\newtheorem{theorem}{Theorem}[section]
\newtheorem{lemma}[theorem]{Lemma}
\newtheorem{proposition}[theorem]{Proposition}

\theoremstyle{definition}
\newtheorem{definition}[theorem]{Definition}
\newtheorem{example}[theorem]{Example}

\theoremstyle{remark}
\newtheorem{remark}[theorem]{Remark}

\newcommand{\Cr}{\mathop{\mathrm{Crit}}\nolimits}

\newcommand{\rank}{\mathop{\mathrm{rank}}\nolimits}

\newcommand{\Ker}{\mathop{\mathrm{Ker}}\nolimits}
\newcommand{\Image}{\mathop{\mathrm{Im}}\nolimits}

\newcommand{\relmiddle}[1]{\mathrel{}\middle#1\mathrel{}}

\begin{document}

\maketitle

\begin{abstract}
We show that the presence of one non-degenerate, non-contractible periodic
orbit of a Hamiltonian on the standard symplectic torus implies the existence of
infinitely many simple non-contractible periodic orbits.
\end{abstract}

\tableofcontents


\section{Introduction}

We consider the torus $\mathbb{T}^{2n}=(\mathbb{R}/\mathbb{Z})^{2n}$ with the standard symplectic form $\omega_{\mathrm{std}}$.
Let $H$ be a Hamiltonian on $\mathbb{T}^{2n}$ and denote by $\mathcal{P}_1(H;\alpha)$ the collection of its
one-periodic orbits with homotopy class $\alpha$.
Our main result is the following theorem.

\begin{theorem}\label{theorem:main}
Let $H\colon S^1\times \mathbb{T}^{2n}\to\mathbb{R}$ be a $($time-dependent$)$ Hamiltonian
having a non-degenerate one-periodic orbit $x$ with non-trivial homotopy class $\alpha$
such that $\mathcal{P}_1(H;\alpha)$ is finite.
Then for every sufficiently large prime $p$, the Hamiltonian $H$ has a simple periodic orbit in the homotopy class $\alpha^p$
and with period either $p$ or $p'$,
where $p'$ is the first prime greater than $p$.
\end{theorem}

In \cite[Theorem 1.1]{G}, B. G\"urel proved that a similar statement holds for closed symplectic manifolds equipped with atoroidal symplectic forms
(see also \cite[Theorem 2.4]{GG3} for a refined vesion of the theorem).
After that, in \cite[Theorem 2.2]{GG3}, V. Ginzburg and B. G\"urel proved the same result for closed toroidally monotone symplectic manifolds under a minor assumption on the ``Euler characteristic."

We note that the standard symplectic form $\omega_{\mathrm{std}}$ on $\mathbb{T}^{2n}$ is not atoroidal.
Moreover, $(\mathbb{T}^{2n},\omega_{\mathrm{std}})$ is not even toroidally monotone.
As pointed out by V. Ginzburg and B. G\"urel in \cite{GG3},
the 2-torus $\mathbb{T}^2$ was the only known example among $\mathbb{T}^{2n}$ ($2n\geq 2$)
that the presence of a non-contractible orbit yields the existence of infinitely many non-contractible orbits.

It is worth pointing out here that Theorem \ref{theorem:main} can be generalized as Theorem \ref{theorem:main2} in Section \ref{section:main}.
Further examples of symplectic manifolds to which Theorem \ref{theorem:main2} applies can be found in Example \ref{example} (see also Remark \ref{remark}).


\section{Preliminaries}\label{section:preliminaries}

In this section, we set conventions and notation and define the filtered Floer--Novikov homology for non-contractible orbits.

\subsection{Conventions and notation}

Let $(M,\omega)$ be a connected closed symplectic manifold.
Let $\mathcal{L}M=C^{\infty}(S^1,M)$ be the space of free loops in $M$ where $S^1=\mathbb{R}/\mathbb{Z}$.
For a free homotopy class $\alpha\in [S^1,M]$,
denote by $\mathcal{L}_{\alpha}M$ the component of $\mathcal{L}M$ with loops representing $\alpha$.

Since a class in $H_1(\mathcal{L}M;\mathbb{Z})$ can be regarded as a linear combination of maps $S^1\times S^1\to M$,
a cohomology class $w\in H^2(M;R)$ defines a cohomology class
\[
	\overline{w}\in H^1(\mathcal{L}M;R)=\mathrm{Hom}(H_1(\mathcal{L}M;\mathbb{Z}),R)
\]
by integrating a 2-form representing $w$ over the maps, where $R=\mathbb{R}$ or $\mathbb{Z}$.
Hence the symplectic form $\omega\in \Omega^2(M)$ and the first Chern class $c_1\in H^2(M;\mathbb{Z})$ of $(M,\omega)$ define cohomology classes
\[
	\overline{[\omega]}\in H^1(\mathcal{L}M;\mathbb{R})=\mathrm{Hom}(H_1(\mathcal{L}M;\mathbb{Z}),\mathbb{R})
\]
and
\[
	\overline{c_1}\in H^1(\mathcal{L}M;\mathbb{Z})=\mathrm{Hom}(H_1(\mathcal{L}M;\mathbb{Z}),\mathbb{Z}),
\]
respectively.

A cohomology class $w\in H^2(M;R)$ is called \textit{aspherical} if $w$ vanishes over $\pi_2(M)$.
Similarly, a cohomology class $w\in H^2(M;R)$ is called \textit{atoroidal} if the cohomology class $\overline{w}$ vanishes over $\pi_1(\mathcal{L}M)$.
We call a closed 2-form $\eta\in\Omega^2(M)$ \textit{atoroidal} if its cohomology class $[\eta]$ is atoroidal.
We note that every atoroidal cohomology class is aspherical.

In addition, a symplectic manifold $(M,\omega)$ is called \textit{toroidally monotone} (resp.\ \textit{toroidally negative monotone}) if we have
\[
	\langle \overline{[\omega]},\pi_1(\mathcal{L}M)\rangle=%
	\lambda\langle \overline{c_1},\pi_1(\mathcal{L}M)\rangle
\]
for some non-negative (resp.\ negative) number $\lambda\in\mathbb{R}$.
A symplectic form $\omega$ is called \textit{$\alpha$-toroidally rational} if the set
$\langle \overline{[\omega]},\pi_1(\mathcal{L}_{\alpha}M)\rangle$ is discrete in $\mathbb{R}$.
Namely, if $\omega$ is $\alpha$-toroidally rational, then there exists a number $h_{\alpha}\in\mathbb{R}$ such that
\[
	\langle \overline{[\omega]},\pi_1(\mathcal{L}_{\alpha}M)\rangle=%
	h_{\alpha}\mathbb{Z}.
\]
We note that every atoroidal symplectic form is $\alpha$-toroidally rational with $h_{\alpha}=0$ for any $\alpha\in [S^1,M]$.
Moreover, every toroidally monotone symplectic manifold has the $\alpha$-toroidally rational symplectic form
with $h_{\alpha}=\lambda c_{1,\alpha}^{\min}$ for any $\alpha\in [S^1,M]$,
where $c_{1,\alpha}^{\min}\in\mathbb{N}$ is the \textit{$\alpha$-minimal first Chern number} given by
\[
	\langle \overline{c_1},\pi_1(\mathcal{L}_{\alpha}M)\rangle=%
	c_{1,\alpha}^{\min}\mathbb{Z}.
\]

In the present paper, we always assume that all Hamiltonians $H$ are one-periodic in time, i.e., $H\colon S^1\times M\to\mathbb{R}$,
and we set $H_t=H(t,\cdot)$ for $t\in S^1=\mathbb{R}/\mathbb{Z}$.
The \textit{Hamiltonian vector field} $X_H\in \mathfrak{X}(M)$ associated to $H$ is defined by
\[
	\iota_{X_H}\omega=-dH.
\]
The \textit{Hamiltonian isotopy} $\{\varphi_H^t\}_{t\in [0,1]}$ associated to $H$ is defined by
\[
	\begin{cases}
		\varphi_H^0=\mathrm{id},\\
		\frac{d}{dt}\varphi_H^t=X_{H_t}\circ\varphi_H^t\quad \text{for all}\ t\in [0,1],
	\end{cases}
\]
and its time-one map $\varphi_H=\varphi_H^1$ is called the \textit{Hamiltonian diffeomorphism} generated by $H$.
Let $\mathcal{P}_k(H;\alpha)$ be the set of $k$-periodic (i.e., defined on $\mathbb{R}/k\mathbb{Z}$) orbits of $H$ representing $\alpha$.
A one-periodic orbit $x\in\mathcal{P}_1(H;\alpha)$ is called \textit{non-degenerate}
if it satisfies $\det\bigl(d\varphi_H(x(0))-\mathrm{id}\bigr)\neq 0$.

Let $K$ and $H$ be two one-periodic Hamiltonians. The composition $K\natural H$ is defined by
\[
	(K\natural H)_t=K_t+H_t\circ (\varphi_K^t)^{-1}.
\]
Then the flow of $K\natural H$ is given by $\varphi_K^t\circ\varphi_H^t$.
For $k\in\mathbb{N}$, we set $H^{\natural k}=H\natural\cdots\natural H$ ($k$ times).
By reparametrizing $H$, we can always assume that $H^{\natural k}$ is one-periodic.
We denote by $x^k$ the $k$th iteration of a one-periodic orbit $x$ of $H$, i.e., $x^k(t)=\varphi_{H^{\natural k}}^t\bigl(x(0)\bigr)$.


\subsection{Floer--Novikov homology}

In this subsection, we define the Floer--Novikov homology for non-contractible periodic orbits (see, e.g., \cite{BPS,BH} for details).

\subsubsection{Action functional}

Let $(M,\omega)$ be a connected closed symplectic manifold.
Let $H\colon S^1\times M\to\mathbb{R}$ be a time-dependent Hamiltonian.
For a free homotopy class $\alpha\in [S^1,M]$,
we fix a reference loop $z\in\alpha$.
We consider the set of pairs $(x,\Pi)$, where $x\in\mathcal{L}_{\alpha}M$
and $\Pi\colon [0,1]\times S^1\to M$ is a path in $\mathcal{L}_{\alpha}M$ between $z$ and $x$.
We set an equivalence relation $\sim$ by $(x,\Pi)\sim (x',\Pi')$ if and only if $x=x'$ and
\[
	\langle\overline{[\omega]},\Pi\#(-\Pi')\rangle=0\quad\text{and}\quad\langle\overline{c_1},\Pi\#(-\Pi')\rangle=0,
\]
where $\Pi\#(-\Pi')$ is a toroidal 2-cycle obtained by gluing $\Pi$ and $\Pi'$ with orientation reversed along the boundaries.
Then the space $\overline{\mathcal{L}_{\alpha}M}$ of such equivalence classes $[x,\Pi]$ is the covering space with structure group
\[
	\Gamma_{\alpha}=\frac{\pi_1(\mathcal{L}_{\alpha}M)}{\Ker\overline{[\omega]}\cap\Ker\overline{c_1}}.
\]
Denote by $\pi\colon\overline{\mathcal{L}_{\alpha}M}\to\mathcal{L}_{\alpha}M$ the covering projection.

We define the \textit{action functional} $\mathcal{A}_H\colon \overline{\mathcal{L}_{\alpha}M}\to \mathbb{R}$ by
\[
	\mathcal{A}_H([x,\Pi])=-\int_{[0,1]\times S^1} \Pi^{\ast}\omega +\int_{0}^{1}H_t\bigl(x(t)\bigr)\, dt.
\]
Since $\pi^{\ast}\overline{[\omega]}=0\in H^1(\overline{\mathcal{L}_{\alpha}M};\mathbb{R})$,
the action functional $\mathcal{A}_H$ is well-defined as a real-valued function.
Note that the critical point set $\Cr(\mathcal{A}_H)$ is equal to $\overline{\mathcal{P}}_1(H;\alpha)=\pi^{-1}\bigl(\mathcal{P}_1(H;\alpha)\bigr)$.
The action functional $\mathcal{A}_H$ is homogeneous with respect to iterations in the sense that
\[
	\mathcal{A}_{H^{\natural k}}([x,\Pi]^k)=k\mathcal{A}_H([x,\Pi]),
\]
where $[x,\Pi]^k=[x^k,\Pi^k]$ is the $k$th iteration of $[x,\Pi]$, and satisfies
\[
	\mathcal{A}_H([x,\Pi]\# A)=\mathcal{A}_H([x,\Pi])-\langle \overline{[\omega]},A\rangle,
\]
for all $A\in\pi_1(\mathcal{L}_{\alpha}M,z)$.
We define the \textit{action spectrum} of $\mathcal{A}_H$ by
\[
	\mathrm{Spec}(H;\alpha)=\mathcal{A}_H\bigl(\overline{\mathcal{P}}_1(H;\alpha)\bigr).
\]


\subsubsection{The filtered Floer--Novikov chain complex}

Let $a$ and $b$ be real numbers such that $-\infty\leq a<b\leq\infty$.
We suppose that the Hamiltonian $H$ satisfies $a,b\not\in\mathrm{Spec}(H;\alpha)$
and \textit{regular}, i.e., all one-periodic orbits $x\in\mathcal{P}_1(H;\alpha)$ are non-degenerate.

Let $J\in\mathcal{J}(M,\omega)$ be a smooth family of $\omega$-compatible almost complex structures.
Consider the Floer differential equation
\begin{equation}\label{eq:1}
	\partial_s u+J(u)\bigl(\partial_t u-X_{H_t}(u)\bigr)=0.
\end{equation}
For a smooth solution $u\colon \mathbb{R}\times S^1\to M$ to \eqref{eq:1}, we define the energy by the formula
\[
	E(u)=\int_0^1\int_{-\infty}^{\infty}\lvert\partial_s u\rvert^2\,dsdt.
\]
Then we have the following:

\begin{lemma}[\cite{S}]
Let $u\colon \mathbb{R}\times S^1\to M$ be a smooth solution to \eqref{eq:1} with finite energy.
\begin{enumerate}
	\item There exist $\bar{x}^{\pm}\in\overline{\mathcal{P}}_1(H;\alpha)$ such that
		\[
			\lim_{s\to\pm\infty}u(s,t)=x^{\pm}(t)\quad \text{and}\quad \lim_{s\to\pm\infty}\partial_s u(s,t)=0,
		\]
		where $\bar{x}^+=[x^+,\Pi^+]$ and $\bar{x}^-=[x^-,\Pi^-]$, and both limits are uniform in the $t$-variable.
		Moreover, we have
		\[
			\Pi^-\# u=\Pi^+.
		\]
	\item The energy identity holds:
		\[
			E(u)=\mathcal{A}_H(\bar{x}^-)-\mathcal{A}_H(\bar{x}^+).
		\]
\end{enumerate}
\end{lemma}

We call a family of almost complex structures \textit{regular}
if the linearized operator for \eqref{eq:1} is surjective for any finite-energy solution of \eqref{eq:1}
in the homotopy class $\alpha$.
We denote by $\mathcal{J}_{\mathrm{reg}}(H;\alpha)$ the space of regular families of almost complex structures.
This subspace is generic in $\mathcal{J}(M,\omega)$ (see \cite{FHS}).
For any $J\in\mathcal{J}_{\mathrm{reg}}(H;\alpha)$ and any pair $\bar{x}^{\pm}\in\overline{\mathcal{P}}_1(H;\alpha)$,
the space
\[
	\mathcal{M}(\bar{x}^-,\bar{x}^+;H,J)=\{\,\text{solution of \eqref{eq:1} satisfying (i)}\,\}
\]
is a smooth manifold whose dimension near such a solution $u$ is given by
the difference of the Conley--Zehnder indices (see \cite{SZ}) of $\bar{x}^-$ and $\bar{x}^+$ relative to $u$.
We denote by $\mathcal{M}^1(\bar{x}^-,\bar{x}^+;H,J)$ the subspace of solutions of relative index one.
For $J\in\mathcal{J}_{\mathrm{reg}}(H;\alpha)$, the quotient $\mathcal{M}^1(\bar{x}^-,\bar{x}^+;H,J)/\mathbb{R}$
is a finite set for any pair $\bar{x}^{\pm}\in\overline{\mathcal{P}}_1(H;\alpha)$.

We set $\overline{\mathcal{P}}_1^a=\{\,\bar{x}\in\overline{\mathcal{P}}_1(H;\alpha)\mid \mathcal{A}_H(\bar{x})<a\,\}$.
We define the chain group of our Floer--Novikov chain complex to be
\[
	\mathrm{CFN}^{[a,b)}(H;\alpha)=\mathrm{CFN}^b(H;\alpha)/\mathrm{CFN}^a(H;\alpha),
\]
where
\[
	\mathrm{CFN}^a(H;\alpha)=\left\{\,\xi=\sum\xi_{\bar{x}}\bar{x}\relmiddle|%
	\begin{aligned}%
		\bar{x}\in\overline{\mathcal{P}}_1^a,\, \xi_{\bar{x}}\in\mathbb{Z}/2\mathbb{Z}\ \text{such that $\forall C\in\mathbb{R}$,}\\%
		\#\{\,\bar{x}\mid\xi_{\bar{x}}\neq 0,\, \mathcal{A}_H(\bar{x})>C\,\}<\infty%
	\end{aligned}\,\right\}.
\]

We define the boundary operator $\partial^{H,J}\colon \mathrm{CFN}^b(H;\alpha)\to \mathrm{CFN}^b(H;\alpha)$ by
\[
	\partial^{H,J}(\bar{x})=\sum \#\left(\mathcal{M}^1(\bar{x},\bar{y};H,J)/\mathbb{R}\right)\,\bar{y}
\]
for a generator $\bar{x}\in \overline{\mathcal{P}}_1^b$.

\begin{theorem}[\cite{F4}]
If $J$ is regular, then the operator $\partial^{H,J}$ is well-defined and satisfies $\partial^{H,J}\circ\partial^{H,J}=0$.
\end{theorem}

The energy identity (ii) implies that $\mathrm{CFN}^a(H;\alpha)$ is invariant under the boundary operator $\partial^{H,J}$.
Thus we get an induced operator $[\partial^{H,J}]$ on the quotient $\mathrm{CFN}^{[a,b)}(H;\alpha)$.

\begin{definition}
The \textit{filtered Floer--Novikov homology group} is defined to be
\[
	\mathrm{HFN}^{[a,b)}(H,J;\alpha)=\Ker{[\partial^{H,J}]}/\Image{[\partial^{H,J}]}.
\]
\end{definition}

\begin{theorem}[\cite{F4,S,SZ}]
If $J_0, J_1\in\mathcal{J}(H;\alpha)$ are two regular almost complex structures, then there exists a natural isomorphism
\[
	\mathrm{HFN}^{[a,b)}(H,J_0;\alpha)\to \mathrm{HFN}^{[a,b)}(H,J_1;\alpha).
\]
\end{theorem}

We refer to $\mathrm{HFN}^{[a,b)}(H;\alpha)=\mathrm{HFN}^{[a,b)}(H,J;\alpha)$ as the Floer--Novikov homology associated to $H$.


\subsubsection{Continuation}

We define the set
\[
	\mathcal{H}^{a,b}(M;\alpha)=\{\,H\colon S^1\times M\to\mathbb{R}\mid a,b\not\in\mathrm{Spec}(H;\alpha)\,\}.
\]

\begin{proposition}[{\cite[Remark 4.4.1]{BPS}}]\label{proposition:nbd}
Every Hamiltonian $H\in\mathcal{H}^{a,b}(M;\alpha)$ has a neighborhood $\mathcal{U}$
such that the Floer--Novikov homology groups $\mathrm{HFN}^{[a,b)}(H',J';\alpha)$, for any regular $H'\in \mathcal{U}$
and any regular almost complex structure $J'\in\mathcal{J}_{\mathrm{reg}}(H';\alpha)$, are naturally isomorphic.
\end{proposition}

Proposition \ref{proposition:nbd} enables us to define the Floer--Novikov homology $\mathrm{HFN}^{[a,b)}(H;\alpha)$
whether $H$ is regular or not.

\begin{definition}\label{definition:degenerate}
For $H\in\mathcal{H}^{a,b}(M;\alpha)$, we define $\mathrm{HFN}^{[a,b)}(H;\alpha)=\mathrm{HFN}^{[a,b)}(\widetilde{H};\alpha)$,
where $\widetilde{H}$ is any regular Hamiltonian sufficiently close to $H$.
\end{definition}

Let $H^+$, $H^-\colon S^1\times M\to\mathbb{R}$ be two Hamiltonians.
We choose regular almost complex structures $J^{\pm}\in\mathcal{J}_{\mathrm{reg}}(H^{\pm};\alpha)$.
We consider a \textit{linear homotopy} $\{H_s\}_{s\in\mathbb{R}}$ from $H^-$ to $H^+$, i.e.,
a smooth homotopy of the form
\[
	(H_s)_t=H_t^-+\beta(s)(H_t^+-H_t^-),
\]
where $\beta\colon\mathbb{R}\to [0,1]$ is a non-decreasing function,
and choose a smooth homotopy $\{J_s\}_{s\in\mathbb{R}}$ from $J^-$ to $J^+$ such that
\[
	(H_s,J_s)=%
	\begin{cases}%
		(H^-,J^-) & \text{if\: $s\ll -1$},\\
		(H^+,J^+) & \text{if\: $s\gg 1$}.
	\end{cases}
\]
We set $H_{s,t}=(H_s)_t$.
Let $\alpha\in [S^1,M]$ be a nontrivial free homotopy class and $a,b\in\mathbb{R}\cup\{\infty\}$ such that $a<b$.
It follows from the energy identity
\[
	E(u)=\mathcal{A}_{H^-}(\bar{x}^-)-\mathcal{A}_{H^+}(\bar{x}^+)+\int_0^1\int_{-\infty}^{\infty}\partial_s H\bigl(s,t,u(s,t)\bigr)\,dsdt
\]
that the Floer--Novikov chain map $\mathrm{CFN}(H^-;\alpha)\to\mathrm{CFN}(H^+;\alpha)$,
defined in terms of the solutions of the equation
\[
	\partial_s u+J_s(u)\bigl(\partial_t u-X_{H_{s,t}}(u)\bigr)=0,
\]
induces a natural homomorphism
\[
	\sigma_{H^+H^-}\colon \mathrm{HFN}^{[a,b)}(H^-;\alpha)\to \mathrm{HFN}^{[a+C,b+C)}(H^+;\alpha),
\]
where $C=C(\{H_s\}_s)$ is a constant given by
\[
	C=\max\left\{\int_0^1\max_{M}\left(H_t^+-H_t^-\right)dt,0\right\}.
\]


\section{Main result}\label{section:main}

In this section, we state a generalized version (Theorem \ref{theorem:main2}) of Theorem \ref{theorem:main} and prove the theorems.

\subsection{Main theorem}

Let $(M,\omega)$ be a closed symplectic manifold.
We recall that an isolated periodic orbit $x$ of $H$ is said to be \textit{homologically non-trivial}
if for some lift $\bar{x}$ of $x$, the local Floer homology $\mathrm{HF}^{\mathrm{loc}}(H,\bar{x})$ of $H$ at $\bar{x}$ is non-zero (see \cite {GG1} for details).
Clearly, a non-degenerate fixed point is homologically non-trivial.

The following theorem is the generalized version of Theorem \ref{theorem:main}.

\begin{theorem}\label{theorem:main2}
Let $H\colon S^1\times M\to\mathbb{R}$ be a Hamiltonian
having an isolated and homologically non-trivial one-periodic orbit $x$ with homotopy class $\alpha$
such that $[\alpha]\neq 0$ in $H_1(M;\mathbb{Z})/\mathrm{Tor}$, and $\mathcal{P}_1(H;[\alpha])$ is finite.
Assume that $\omega$ is $\alpha$-toroidally rational and
\[
	\langle \overline{[\omega]},\pi_1(\mathcal{L}_{\alpha^p}M)\rangle=%
	p\langle \overline{[\omega]},\pi_1(\mathcal{L}_{\alpha}M)\rangle
\]
holds for every sufficiently large prime $p$.
Then for every sufficiently large prime $p$, the Hamiltonian $H$ has a simple periodic orbit in the homotopy class $\alpha^p$
and with period either $p$ or $p'$,
where $p'$ is the first prime greater than $p$.
\end{theorem}

\begin{example}\label{example}
(i)\: One of the most important examples of symplectic manifolds meeting the requirements of Theorem \ref{theorem:main2} is the torus $(\mathbb{T}^{2n},\omega_{\mathrm{std}})$ (see Lemma \ref{lemma:key} for details).

(ii)\: It is clear that we can apply Theorem \ref{theorem:main2} for every closed symplectic manifold $N$ with an atoroidal symplectic form $\omega_N$ since we have
$\langle \overline{[\omega_N]},\pi_1(\mathcal{L}_{\alpha}N)\rangle=0$ for all $\alpha\in [S^1,N]$.
Moreover, the product $(M\times N,\omega\oplus\omega_N)$ is also an example
where $(M,\omega)$ is a closed symplectic manifold satisfying the assumptions in Theorem \ref{theorem:main2} and $\omega_N$ is atoroidal.
Indeed, we have
\begin{align*}
	\left\langle \overline{[\omega\oplus\omega_N]},\pi_1\left(\mathcal{L}_{(\alpha,\alpha_N)}(M\times N)\right)\right\rangle%
	&=\langle \overline{[\omega]},\pi_1(\mathcal{L}_{\alpha}M)\rangle%
	+\langle \overline{[\omega_N]},\pi_1(\mathcal{L}_{\alpha_N}N)\rangle\\
	&=\langle \overline{[\omega]},\pi_1(\mathcal{L}_{\alpha}M)\rangle.
\end{align*}
\end{example}


\subsection{Proof of Theorem \ref{theorem:main}}

In this subsection, we prove Theorem \ref{theorem:main} by assuming that Theorem \ref{theorem:main2} holds.
We call an element $e\in\mathbb{Z}^m$ \textit{primitive}
if $e\neq kf$ for any positive integer $k\in\mathbb{N}$ and any $f\in\mathbb{Z}^m\setminus\{e\}$.

Now we consider the torus $\mathbb{T}^{2n}=(\mathbb{R}/\mathbb{Z})^{2n}$ with the standard symplectic form $\omega_{\mathrm{std}}$.
Let $\alpha\in [S^1,\mathbb{T}^{2n}]\cong\mathbb{Z}^{2n}$ be a non-trivial homotopy class.
Then there exist a positive integer $h_{\alpha}\in\mathbb{N}$ and a primitive vector $e_{\alpha}\in\mathbb{Z}^{2n}$ such that $\alpha=h_{\alpha}e_{\alpha}$.
For $\alpha=0$, we set $h_{\alpha}=0$ and $e_{\alpha}=0$.
Then the following lemma is crucial.

\begin{lemma}\label{lemma:key}
We have
\begin{align*}
	\left\langle \overline{[\omega_{\mathrm{std}}]},\pi_1\left(\mathcal{L}_{\alpha}\mathbb{T}^{2n}\right)\right\rangle%
	&=h_{\alpha}\left\{\,\sum_{i=1}^n%
	\begin{vmatrix}
		e_{2i-1} &b_{2i-1} \\
		e_{2i} &b_{2i}
	\end{vmatrix}
	\relmiddle|(b_1,\ldots,b_{2n})\in\mathbb{Z}^{2n}\,\right\}\\
	&=%
	h_{\alpha}\mathbb{Z},
\end{align*}
where $e_{\alpha}=(e_1,\ldots,e_{2n})\in\mathbb{Z}^{2n}$.

In particular, we have
\[
	\left\langle \overline{[\omega_{\mathrm{std}}]},\pi_1\left(\mathcal{L}_{\alpha^k}\mathbb{T}^{2n}\right)\right\rangle%
	=k\left\langle \overline{[\omega_{\mathrm{std}}]},\pi_1\left(\mathcal{L}_{\alpha}\mathbb{T}^{2n}\right)\right\rangle
\]
for all $\alpha\in [S^1,\mathbb{T}^{2n}]$ and $k\in\mathbb{N}$.
\end{lemma}

\begin{proof}
Since the torus $\mathbb{T}^{2n}$ is an $H$-space with a commutative multiplication $\mu$,
it has the associative Pontryagin product $p$, i.e., the bilinear map
\[
	H_i(\mathbb{T}^{2n};\mathbb{Z})\times H_j(\mathbb{T}^{2n};\mathbb{Z})\xrightarrow{\times}H_{i+j}(\mathbb{T}^{2n}\times\mathbb{T}^{2n};\mathbb{Z})\xrightarrow{\mu_{\ast}}H_{i+j}(\mathbb{T}^{2n};\mathbb{Z}),
\]
where $\times$ is the cross product and  $\mu_{\ast}$ is the homomorphism induced by $\mu$.
It is known that the Pontryagin ring $(H_{\ast}(\mathbb{T}^{2n};\mathbb{Z}),p)$ is the exterior algebra
$\left(\bigwedge_{\mathbb{Z}}[x_1,\ldots,x_{2n}],\wedge\right)$ with $\lvert x_i\rvert=1$.
Let $h\colon [S^1,\mathbb{T}^{2n}]\to H_1(\mathbb{T}^{2n};\mathbb{Z})$ be a map given by the formula $h(\alpha)=x_{\ast}([S^1])$
where $x$ is a loop representing $\alpha$ and $[S^1]$ is the fundamental class of $S^1$.

On the other hand, now every element $A$ in $\pi_1\bigl(\mathcal{L}_{\alpha}\mathbb{T}^{2n}\bigr)$
is represented by a map $S^1\times S^1\to \mathbb{T}^{2n}$ such that
$S^1\times\{\ast\}\to \mathbb{T}^{2n}$ represents $\alpha$.
Let $\beta\in [S^1,\mathbb{T}^{2n}]$ be the homotopy class of the map $\{\ast\}\times S^1\to \mathbb{T}^{2n}$.
Hence $A$ is identified with an element $p(h(\alpha),h(\beta))$ in $H_2(\mathbb{T}^{2n};\mathbb{Z})$.
Let us denote $\alpha=(a_1,\ldots,a_{2n})\in [S^1,\mathbb{T}^{2n}]\cong\mathbb{Z}^{2n}$.
Then we have
\begin{align*}
	\left\langle \overline{[\omega_{\mathrm{std}}]},\pi_1\left(\mathcal{L}_{\alpha}\mathbb{T}^{2n}\right)\right\rangle
	&=\left\{\,\left\langle [\omega_{\mathrm{std}}],p(h(\alpha),h(\beta))\right\rangle\relmiddle|\beta\in [S^1,\mathbb{T}^{2n}]\,\right\}\\
	&=\left\{\,\left\langle [\omega_{\mathrm{std}}],%
	\left(%
	\begin{array}{c}
		a_1\\ \vdots\\ a_{2n}
	\end{array}
	\right)%
	\wedge%
	\left(%
	\begin{array}{c}
		b_1\\ \vdots\\ b_{2n}
	\end{array}
	\right)%
	\right\rangle\relmiddle|(b_1,\ldots,b_{2n})\in\mathbb{Z}^{2n}\,\right\}.
\end{align*}
Here the exterior product is computed as
\begin{align*}
	\left(%
	\begin{array}{c}
		a_1\\ \vdots\\ a_{2n}
	\end{array}
	\right)%
	\wedge%
	\left(%
	\begin{array}{c}
		b_1\\ \vdots\\ b_{2n}
	\end{array}
	\right)%
	&=\left(%
	\begin{vmatrix}
		a_1 &b_1 \\
		a_2 &b_2
	\end{vmatrix}
	,-%
	\begin{vmatrix}
		a_1 &b_1 \\
		a_3 &b_3
	\end{vmatrix}
	,%
	\begin{vmatrix}
		a_1 &b_1 \\
		a_4 &b_4
	\end{vmatrix}
	,\cdots,%
	\begin{vmatrix}
		a_{2n-1} &b_{2n-1} \\
		a_{2n} &b_{2n}
	\end{vmatrix}
	\right)\\
	&\in\mathbb{Z}^{n(2n-1)}\cong H_2(\mathbb{T}^{2n};\mathbb{Z}).
\end{align*}
Hence if we denote a point in $\mathbb{T}^{2n}$ by $(p_1,q_1,\ldots,p_n,q_n)$, then for each $i=1,\ldots,n$, we have
\begin{align*}
	\left\langle [dp_i\wedge dq_i],%
	\left(%
	\begin{array}{c}
		a_1\\ \vdots\\ a_{2n}
	\end{array}
	\right)%
	\wedge%
	\left(%
	\begin{array}{c}
		b_1\\ \vdots\\ b_{2n}
	\end{array}
	\right)%
	\right\rangle
	&=%
	\begin{vmatrix}
	a_{2i-1} &b_{2i-1} \\
	a_{2i} &b_{2i}
	\end{vmatrix}
	\int_{\mathbb{T}^2}dp_i\wedge dq_i\\
	&=%
	\begin{vmatrix}
	a_{2i-1} &b_{2i-1} \\
	a_{2i} &b_{2i}
	\end{vmatrix}
	=h_{\alpha}%
	\begin{vmatrix}
		e_{2i-1} &b_{2i-1} \\
		e_{2i} &b_{2i}
	\end{vmatrix}.
\end{align*}
Thus we conclude that
\begin{align*}
	&\left\{\,\left\langle [\omega_{\mathrm{std}}],%
	\left(%
	\begin{array}{c}
		a_1\\ \vdots\\ a_{2n}
	\end{array}
	\right)%
	\wedge%
	\left(%
	\begin{array}{c}
		b_1\\ \vdots\\ b_{2n}
	\end{array}
	\right)%
	\right\rangle\relmiddle|(b_1,\ldots,b_{2n})\in\mathbb{Z}^{2n}\,\right\}\\
	&=h_{\alpha}\left\{\,\sum_{i=1}^n%
	\begin{vmatrix}
		e_{2i-1} &b_{2i-1} \\
		e_{2i} &b_{2i}
	\end{vmatrix}
	\relmiddle|(b_1,\ldots,b_{2n})\in\mathbb{Z}^{2n}\,\right\}.
\end{align*}
Hence we have the desired equality.
\end{proof}

\begin{proof}[Proof of Theorem \ref{theorem:main}]
We assume that Theorem \ref{theorem:main2} holds.
Since $[S^1,\mathbb{T}^{2n}]$ and $H_1(\mathbb{T}^{2n};\mathbb{Z})/\mathrm{Tor}$ are both isomorphic to $\mathbb{Z}^{2n}$,
the assumptions on the free homotopy class $\alpha$ in both theorems are equivalent.
Then Lemma \ref{lemma:key} immediately shows that Theorem \ref{theorem:main}.
\end{proof}

\begin{remark}\label{remark}
(i)\: The latter part of Lemma \ref{lemma:key} holds
for any connected closed symplectic manifold having the Pontryagin product such that the Pontryagin ring is the exterior algebra.
However, it is known that every connected compact symplectic Lie group (i.e., a Lie group with a left-invariant symplectic form)
is a torus \cite{C}.

(ii)\: The proof of Lemma \ref{lemma:key} shows that we can apply Theorem \ref{theorem:main2} for the torus $\mathbb{T}^{2n}=(\mathbb{T}^2)^n$ with a symplectic form
\[
	\omega_A=\sum_{j=1}^n A_j\mathrm{pr}_j^{\ast}\omega_{\mathbb{T}^2},
\]
where $\omega_{\mathbb{T}^2}$ is the standard symplectic form on the 2-torus $\mathbb{T}^2=(\mathbb{R}/\mathbb{Z})^2$,
the map $\mathrm{pr}_j\colon (\mathbb{T}^2)^n\to\mathbb{T}^2$ is the $j$th projection and
$A=(A_1,\ldots,A_n)$ is an $n$-tuple of real numbers satisfying
\[
	\rank_{\mathbb{Z}}\langle A_1,\ldots,A_n\rangle_{\mathbb{Z}}=1.
\]
\end{remark}


\subsection{Proof of Theorem \ref{theorem:main2}}

The proof of Theorem \ref{theorem:main2} is inspired by the argument by B. G\"urel \cite{G}.

\begin{proof}
Since $\mathcal{P}_1(H;[\alpha])$ is finite, there exist finitely many distinct homotopy classes $\alpha_i\in [S^1,M]$
representing $[\alpha]\in H_1(M;\mathbb{Z})/\mathrm{Tor}$ such that every $x\in\mathcal{P}_1(H;[\alpha])$ is contained in one of $\alpha_i$'s.
As in \cite{GG3}, one can show that for every sufficiently large prime $p$, the classes $\alpha_i^p$ are all distinct.

Fix a reference loop $z\in\alpha$ and
choose the iterated loop $z^p$ as the reference loop for $\alpha^p$.
Denote by $x_i$ the elements of $\mathcal{P}_1(H;\alpha)$.
We note that every sufficiently large prime $p$ is \textit{admissible} in the sense of \cite{GG1} for all orbits $x_i$
(i.e., $\lambda^p\neq 1$ for all eigenvalues $\lambda\neq 1$ of $(d\varphi_H)_{x_i}\colon T_{x_i}M\to T_{x_i}M$).
Then under such iterations of $H$, the orbit $x$ remains isolated and
\[
	\mathrm{HF}^{\mathrm{loc}}(H^{\natural p},\bar{x}^p)\cong \mathrm{HF}^{\mathrm{loc}}(H,\bar{x}),
\]
where $\bar{x}=[x,\Pi]\in\overline{\mathcal{L}_{\alpha}M}$ is some lift of $x$
and $\bar{x}^p=[x^p,\Pi^p]\in\overline{\mathcal{L}_{\alpha^p}M}$.
Now we have $\mathrm{HF}^{\mathrm{loc}}(H^{\natural p},\bar{x}^p)\neq 0$ since $\mathrm{HF}^{\mathrm{loc}}(H,\bar{x})\neq 0$.

Assume that $H$ has no simple $p$-periodic orbit in $\alpha^p$.
Since $p$ is prime, all $p$-periodic orbits in $\alpha^p$ are the $p$th iterations of one-periodic orbits in $\alpha$.
Hence we have
\begin{equation}\label{eq:P}
	\mathcal{P}_1(H^{\natural p};\alpha^p)=\left\{\,y^p\relmiddle| y\in\mathcal{P}_1(H;\alpha)\,\right\}.
\end{equation}
Now we claim that we have
\begin{equation}\label{eq:S}
	\mathrm{Spec}(H^{\natural p};\alpha^p)=p\mathrm{Spec}(H;\alpha)
\end{equation}
for every sufficiently large $p$.
Indeed, it is clear that $\mathrm{Spec}(H^{\natural p};\alpha^p)\supset p\mathrm{Spec}(H;\alpha)$.
In order to show the opposite side, let us choose $c\in\mathrm{Spec}(H^{\natural p};\alpha^p)$.
By \eqref{eq:P}, then there exist $\bar{y}\in\overline{\mathcal{P}}_1(H;\alpha)$ and $A\in\pi_1(\mathcal{L}_{\alpha^p}M)$ such that
\[
	c=\mathcal{A}_{H^{\natural p}}(\bar{y}^p\# A)=p\mathcal{A}_H(\bar{y})-\langle \overline{[\omega]},A\rangle.
\]
Now we choose $p$ so large that
\[
	\langle \overline{[\omega]},\pi_1(\mathcal{L}_{\alpha^p}M)\rangle=%
	p\langle \overline{[\omega]},\pi_1(\mathcal{L}_{\alpha}M)\rangle
\]
holds.
Then there exists $B\in\pi_1(\mathcal{L}_{\alpha}M)$ such that
$\langle \overline{[\omega]},A\rangle=p\langle \overline{[\omega]},B\rangle$.
Therefore
\[
	c=p\left(\mathcal{A}_H(\bar{y})-\langle \overline{[\omega]},B\rangle\right)\in p\mathrm{Spec}(H;\alpha).
\]
It implies that $\mathrm{Spec}(H^{\natural p};\alpha^p)\subset p\mathrm{Spec}(H;\alpha)$.

By adding a constant to the Hamiltonian $H$, we can assume that the action of the lift $\bar{x}$ is
$\mathcal{A}_H(\bar{x})=0$.
Hence for all $p$, we have
\[
	\mathcal{A}_{H^{\natural p}}(\bar{x}^p)=p\mathcal{A}_H(\bar{x})=0.
\]
Since $\mathcal{P}_1(H;[\alpha])$ is finite and $\omega$ is $\alpha$-toroidally rational,
we can choose $a>0$ so small that
\[
	[-a,a)\cap\mathrm{Spec}(H;\alpha)=\{0\}.
\]
By \eqref{eq:S}, we then have
\[
	[-pa,pa)\cap\mathrm{Spec}(H^{\natural p};\alpha^p)=\{0\}.
\]
Hence zero is the only critical value of $\mathcal{A}_{H^{\natural p}}$ in $[-pa,pa)$.
Therefore,
\[
	\mathrm{HFN}^{[-pa,pa)}(H^{\natural p};\alpha^p)=\mathrm{HF}^{\mathrm{loc}}(H^{\natural p},\bar{x}^p)\oplus\cdots,
\]
where the dots represent the contributions of the local Floer homology groups of some lifts $\bar{x}_i\in\overline{\mathcal{P}}_1(H;\alpha)$
with zero action, of the other one-periodic orbits such that $x_i^p\in\alpha^p$.
We set
\[
	C=\max\left\{\int_{S^1}\max_M H_t\,dt,0\right\}+\max\left\{-\int_{S^1}\min_M H_t\,dt,0\right\}.
\]
Now we assume $p$ so large that $pa>6C(p'-p)$ since $p'-p=o(p)$ as $p\to\infty$ (see, e.g., \cite{BHP}).
Choose $K>0$ such that
\[
	pa-4C(p'-p)<K<pa-2C(p'-p).
\]
Then we have
\begin{equation}\label{eq:inequality}
	-pa<-K<-K+2C(p'-p)<0<K<K+2C(p'-p)<pa,
\end{equation}
and
\[
	-p'a<-K+C(p'-p)<0<K+C(p'-p)<p'a.
\]
We set $\delta=C(p'-p)$.
Now we have the following commutative diagram:
\[
	\xymatrix{
	\mathrm{HFN}^{[-K,K)}(H^{\natural p};\alpha^p) \ar[d]_{\sigma_{H^{\natural p'}H^{\natural p}}} \ar[rrd]^{\cong} & & \\
	\mathrm{HFN}^{[-K+\delta,K+\delta)}(H^{\natural p'};\alpha^p) \ar[rr]_{\sigma_{H^{\natural p}H^{\natural p'}}} & & \mathrm{HFN}^{[-K+2\delta,K+2\delta)}(H^{\natural p};\alpha^p)
	}
\]
Here the map $\sigma_{H^{\natural p'}H^{\natural p}}$ (resp.\ $\sigma_{H^{\natural p}H^{\natural p'}}$) is induced by
a linear homotopy from $H^{\natural p}$ to $H^{\natural p'}$ (resp.\ from $H^{\natural p'}$ to $H^{\natural p}$),
and the diagonal map is an isomorphism induced by the natural quotient-inclusion map (see \eqref{eq:inequality}).
Combining with $\mathrm{HFN}^{[-K,K)}(H^{\natural p};\alpha^p)\neq 0$, we conclude that
\[
	\mathrm{HFN}^{[-K+\delta,K+\delta)}(H^{\natural p'};\alpha^p)\neq 0.
\]
Thus $H$ has a $p'$-periodic orbit $y$ in the homotopy class $\alpha^p$, and hence in the homology class $p[\alpha]$.

Now it is enough to show that $y$ is simple.
Arguing by contradiction, we assume that $y$ is not simple.
Since $p'$ is prime, $y$ is the $p'$th iteration of a one-periodic orbit in the homology class $p[\alpha]/p'\in H_1(M;\mathbb{Z})/\mathrm{Tor}$.
Since $p/p'$ is not an integer, this contradicts the assumption that $[\alpha]\neq 0\in H_1(M;\mathbb{Z})/\mathrm{Tor}$.
\end{proof}


\subsection*{Acknowledgement}

The author would like to express his sincere gratitude to Urs Frauenfelder for many fruitful discussions.
The author is also grateful to Viktor Ginzburg, Ba\c{s}ak G\"urel, Morimichi Kawasaki, Yoshihiko Mitsumatsu and the author's advisor Takashi Tsuboi for many valuable comments.
This work was carried out while the author was visiting the University of Augsburg.
The author would like to thank the institute for its warm hospitality and support.


\bibliographystyle{amsart}

\end{document}